\documentclass[amsart]{article}
\usepackage[left=2.5cm, right=2.5cm, top=3cm, bottom=3cm]{geometry} 
\usepackage[utf8]{inputenc}
\usepackage[english]{babel}
\usepackage{amssymb}
\usepackage{graphicx}
\usepackage[tbtags]{amsmath}
\usepackage{amsfonts}
\usepackage{mathrsfs}
\usepackage{xcolor}
\usepackage{commath}
\usepackage{amsthm}
\usepackage{lineno}

\newtheorem{theorem}{\quad Theorem}[section]

\newtheorem{definition}[theorem]{\quad Definition}

\newtheorem{corollary}[theorem]{\quad Corollary}

\newtheorem{lemma}[theorem]{\quad Lemma}

\newtheorem{proposition}[theorem]{\quad Proposition}

\newtheorem{thm}{Theorem}[section]

\newtheorem{deff}[thm]{Definition}

\title{On a shape derivative formula for the Robin $p$-Laplace eigenvalue}
\author{Ardra A \footnote{The author was supported by the Department of Science and Technology INSPIRE Fellowship.}\;\and Mohan Mallick \and Sarath Sasi \footnote{Corresponding author}}

\begin{document}
	\maketitle
	
	\begin{abstract}
		We obtain shape derivative formulae for the first eigenvalue of the Robin $p$-Laplace operator. This result is used to study the variation of the first eigenvalue with respect to perturbations of the domain. In particular, we prove that for large values of the boundary parameter, the first eigenvalue is monotonic with respect to domain inclusion for smooth domains. 
	\end{abstract}
 
	\section{Introduction}\label{Introduction section}

	Consider the following eigenvalue problem with the Robin boundary condition:
	\begin{equation}
	\label{Robin p-Laplacian}
	\left\{
	\begin{split}
	\Delta_p u+\lambda  |u|^{p-2}u&=0\quad \rm{in} ~\Omega,\\
	|\nabla u|^{p-2}\frac{\partial u}{\partial\eta}+\beta|u|^{p-2}u&=0\quad\rm{on}~\partial\Omega,
	\end{split}
	\right.
	\end{equation}
	where $\Delta_p u = \mbox{div} (|\nabla u|^{p-2} \nabla u)$ is the $p$-Laplace operator for $1<p<\infty,$ $ \Omega $ is a bounded domain in $\mathbb{R}^n$ with smooth boundary, $ \eta $ denotes the outward unit normal, and $ \beta $ is a positive real constant. The Dirichlet and Neumann boundary conditions can be considered as special cases of the Robin boundary condition with $\beta $ going to infinity and $\beta = 0, $ respectively (see \cite{le2006eigenvalue}). In the case when $p=2$, problem (\ref{Robin p-Laplacian}) reduces to the standard Robin Laplace eigenvalue problem.   
                      	
    \begin{definition} \rm{\cite{le2006eigenvalue}
		Any $(u,\lambda)\in W^{1,p}(\Omega) \times \mathbb{R}$ satisfying $$ \int_{\Omega} |\nabla u |^{p-2}\nabla u \cdot \nabla v ~dx + \int_{\partial \Omega} \beta |u|^{p-2}uv ~ds = \lambda\int_{\Omega} |u|^{p-2}uv ~dx$$ for all $v\in W^{1,p}(\Omega)$ is said to be a weak solution of (\ref{Robin p-Laplacian}).} 
	\end{definition}
	
	If $(u,\lambda)$ is a weak solution of (\ref{Robin p-Laplacian}) with $u \neq 0$, then $\lambda$ is known as an eigenvalue, and $u$ is a corresponding eigenfunction. We can use Ljusternik–Schnirelman (L-S) theory to prove the existence  of a sequence of eigenvalues $\{\lambda_m\}$ such that $0 \leq\lambda_1 \leq\lambda_2\leq ... \leq \lambda_m \leq...,$ with $\lambda_m\longrightarrow\infty$ as $m\longrightarrow\infty.$ It is not known whether this exhausts the spectrum when $p\neq 2,$ but it is known that the set of eigenvalues is closed (see \cite{le2006eigenvalue}).
	
Consider the functional $J(z)= \frac{\int_{\Omega} |\nabla z |^p ~dx + \beta \int_{\partial \Omega}  |z|^p ~ds}  {\int_{\Omega}  |z|^p ~dx}$ defined on $W^{1,p}(\Omega)$. The first eigenvalue $\lambda_1$ of \eqref{Robin p-Laplacian} is given by the characterization 
\begin{equation}
\label{Rayleigh}
\lambda_1(\Omega)  = \displaystyle{\inf_{\substack{ z \in W^{1,p}(\Omega) \\ z \neq 0}}} J(z).
\end{equation}
In \cite{le2006eigenvalue}, the authors proved that if $\Omega $ has a $C^2$ boundary, then $\lambda_1(\Omega)$ is simple and isolated. Moreover, the eigenfunctions associated with all eigenvalues except $\lambda_1(\Omega)$ change sign, and there exists a positive constant $\alpha$ such that $u \in C^{1,\alpha}(\overline{\Omega})$ (see \cite{le2006eigenvalue} for more details). Even in the case $p=2,$  many basic properties of $\lambda_1(\Omega)$ are not well-understood. For instance, in the case of the Dirichlet Laplacian, it is known that if $\Omega_1 \subseteq \Omega_2,$ then the first Dirichlet eigenvalue $\lambda_1^D(\Omega)$ satisfies $\lambda_1^D(\Omega_2) \leq \lambda_1^D(\Omega_1).$ However, this is not true in general for the Robin eigenvalue problem (see \cite{giorgi2005monotonicity} for a counterexample), and the domain monotonicity property is not well understood. This is discussed in more detail later in this section. 
	
There have been many studies on the domain dependence of energy functionals associated with $p$-Laplace differential equations with Dirichlet boundary  conditions (see \cite{melian2001perturbation}, \cite{ii2001placement}, \cite{bobkov2020qualitative}, \cite{lamberti2003differentiability}, and \cite{roppongi1994hadamard}).
For smooth domains, the standard technique used  in many of these studies is Hadamard's method of variations. In this technique, we consider a family of diffeomorphisms
$T_t\in \mathcal{C}^1(\Omega, \mathbb{R}^n)$,  $|t|\in \mathbb{R}$ sufficiently small. We define
$$T_t(x)=x+tv(x)+o(t),$$
where  $v=(v_1(x),v_2(x),..., v_n(x)) $ is a $C^1(\overline{\Omega})$ vector field and $o(t)$ denotes all the terms with $\frac{o(t)}{t}\to 0$ as $t\to 0.$
Now letting $\Omega_t=T_t(\Omega)$,  the perturbation of energy functionals can be studied by analyzing the dependence of the functionals on $\Omega_t$ with respect to $t$. 
Using this approach, the domain derivative for the first eigenvalue of the Dirichlet $p$-Laplacian  was obtained by Garc\'{i}a Meli\'{a}n and Sabina De Lis in \cite{melian2001perturbation}. They prove that for bounded $\mathcal{C}^{2,\alpha}$ domains $\lambda^D_1(t):=\lambda^D_1(\Omega_t)$ is differentiable with respect to $t$ at $0,$ and the derivative is given by
\begin{equation}
\label{GarciaMelianFormula}
\dot{\lambda}^{D}_1(0)=-(p-1) \int_{\partial \Omega}\abs{\frac{\partial u}{\partial \eta}}^p(v \cdot \eta)~ds,
\end{equation} where $u$ is the positive normalized Dirichlet eigenfunction so that $\int_{\Omega}  |u|^p ~dx=1.$ A similar formula was derived by Roppongi in \cite{roppongi1994hadamard}. As has been noted by Garabedian and Schiffer in \cite{garabedian1953convexity},  this formula for the case $p=2$ was first obtained by Hadamard in  \cite{hadamard1908memoire}.

The formula given in \eqref{GarciaMelianFormula} has been used to study various questions regarding the behavior of eigenvalues and the structure of eigenfunctions. In particular, in \cite{ii2001placement}, \cite{anoop2018strict}, and \cite{kesavan2003two}, the authors have used domain derivatives to study obstacle problems. In \cite{anoop2016structure}, Anoop et al. used them to prove the nonradiality of second eigenfunctions of Dirichlet $p$-Laplacian.

A similar shape derivative formula was obtained for the Robin Laplace operator by Bandle and Wagner in \cite{bandle2015second}. In this paper, the authors use Hadamard's perturbation techniques to get the following domain derivative formula for $\lambda_1 (t):=\lambda_1(\Omega_t);$
\begin{equation}
\label{Robin Laplace lambda dot 0}
    \dot{\lambda}_1(0)=\int_{\partial \Omega}  \{ \left[|\nabla u|^2 - \lambda_1 (0) u^2 + \beta (n-1)Hu^2 - 2 \beta^2 u^2  \right] (v \cdot \eta) \}  ~ds,
\end{equation}
where $\Omega \subseteq \mathbb{R}^n$ is a bounded domain in the H$\ddot{\text{o}}$lder class $C^{2,\alpha}$, $u$ is the corresponding normalized eigenfunction such that $\int_{\Omega}  |{u}(t)|^2 J(t) ~dx=1,$ with $J(t)$ being the determinant of the Jacobian matrix corresponding to the transformation $T_t,$  $H$ denotes the mean curvature of the boundary, $v$ is a $C^{2,\alpha}$ vector field. We would like to  note that a Hadmard variational formula for the Robin Laplace semi-linear problem has been studied by Osawa in \cite{osawa1992hadamard}.

 In this paper, we obtain shape derivative formulae for the first eigenvalue of the Robin $p$-Laplace operator. The boundary integral in the characterization of the first eigenvalue makes the computation significantly harder than the Dirichlet case.  We follow the approach in \cite{bandle2015second}, however, the nonlinearity of the operator makes the problem highly nontrivial. For instance, the lack of $C^2$ regularity for the first eigenfunction of the $p$-Laplace operator introduces an additional layer of difficulty. Now we state our main result.

\begin{theorem}
	\label{Final formulae for lambda dot 0}
Let $\Omega \subseteq \mathbb{R}^n$ be a bounded domain in the H$\ddot{\text{o}}$lder class $C^{2,\alpha},$ $\Omega_t$ be a family of perturbations with $ \Omega_t :=\{y=x+tv(x)+o(t):x\in \Omega, |t| \text{ sufficiently small}\}$ where $v=(v_1(x),v_2(x),..., v_n(x)) $ is a $C^{2,\alpha}$ vector field and $H$ denotes the mean curvature of $\partial \Omega .$ Then for the eigenvalue $\lambda_1(t):=\lambda_1(\Omega_t)$ of \eqref{Robin p-Laplacian},with $\Omega=\Omega_t,$ we have the following equivalent formulae for $\dot{\lambda}_1(0):$ 
\begin{enumerate}
	
	\item 
	$
	\!
	\begin{aligned} [t]
	\dot{\lambda}_1(0)&=\int_{\partial \Omega} \{\left[|\nabla u|^p - \lambda_1 (0) |u|^p + \beta |u|^p (n-1)H + p \beta u \left( |u|^{p-2} \nabla u \cdot \eta \right) \right] (v \cdot \eta)  \} ~ds. 
	\end{aligned}
	$
	\item 
	$
	\!
	\begin{aligned} [t]
	\dot{\lambda}_1(0)&=\int_{\partial \Omega}  \{ \left[|\nabla u|^p - \lambda_1 (0) |u|^p + \beta |u|^p (n-1)H - p \beta^2 \frac{|u|^{2p-2}}{|\nabla u|^{p-2}}  \right] (v \cdot \eta) \}  ~ds. 
	\end{aligned}
	$ 
\end{enumerate}	
\end{theorem}
\noindent In the case $p=2,$ these expressions for $\dot{\lambda}_1(0)$ reduce to \eqref{Robin Laplace lambda dot 0}. 

As mentioned at the beginning, for the first Robin eigenvalue, there is no general domain monotonicity. A counterexample for the case $p=2$ can be found in \cite{giorgi2005monotonicity}, where the authors have constructed a sequence $\Omega_k$ for which  $B(0,1)\subset \Omega_k \subset B(0,2)$ such that $\lambda_1(\Omega_k) > \lambda_1(B(0,1))>\lambda_1(B(0,2))>0,$ for large values of $k.$ Even in the case of the Robin Laplacian, the literature on domain monotonicity is rather limited. For a sufficiently smooth and bounded domain $\Omega$, Payne and Weinberger proved in \cite{payne1957lower} that $\lambda_1 (\Omega) \geq \lambda_1 (B)$, for a ball $B$ circumscribing $\Omega$ in two and three dimensions. Giorgi and Smits have shown partial monotonicity for the principal eigenvalue when the second domain is a ball (see Theorem 1 in \cite{giorgi2005monotonicity}). They have also proved monotonicity results for a convex domain containing a ball (see Theorem 2 and Corollary 3 of \cite{giorgi2005monotonicity}) and for two convex domains with a ball contained in between (see Corollary 4 of \cite{giorgi2005monotonicity}). In case of dilations of the domain, $\lambda_1(t \Omega) \leq \lambda_1(\Omega), $ for every $t \geq 1$ (see Chapter 4 in \cite{henrot2017shape}).
	
	 For the Robin $p$-Laplacian, Gavitone and Trani have proved some domain monotonicity results for the principal eigenvalue. They have proved these results for a more general class of anisotropic operators. They have shown that $\lambda_1(W)\leq\lambda_1(G), $ for a bounded open set $G\in \mathbb{R}^n$ with $C^{1,\alpha}$ boundary, $\alpha \in (0,1),$ contained in a Wulff shape $W$ (see Theorem 5.1 in \cite{gavitone2018first}). If, in addition, $G$ is convex such that $W \subset G,$ then $\lambda_1(G)\leq\lambda_1(W)$ (see Theorem 5.2 in \cite{gavitone2018first}). Domain monotonicity has also been proved for two convex sets containing a Wulff shape in between (see Corollary 5.1 in \cite{gavitone2018first}). 

We use the shape derivative formulae given in Theorem \ref{Final formulae for lambda dot 0} to obtain some domain monotonicity results. In the next theorem, we look at volume increasing (decreasing) perturbations  in the case when $\Omega$ is a ball and obtain conditions for domain monotonicity. 

	 \begin{theorem}
         \label{Ball case}
         Let $\Omega  = B(0,R)$ in $\mathbb{R}^n.$ 
        \begin{enumerate}
             \item [(i)] If $\beta \geq \left(\frac{n-1}{R(p-1)}\right)^{p-1},$ then $\dot{\lambda}_1(0)<0 \; (>0)$ for a smooth perturbation $v$ of the domain  with $\int_{\partial\Omega}v \cdot \eta ~ds >0 \; (<0)$ on $\partial \Omega$.

             \item [(ii)] Let $R \geq 1,\beta \geq 1$ and $p \geq n.$ Then $\dot{\lambda}_1(0)\leq 0 \; (\geq 0)$ for a smooth $v$  perturbation of the domain with $\int_{\partial\Omega}v \cdot \eta ~ds >0 \; (<0)$ on $\partial \Omega$.
         \end{enumerate} 
\end{theorem}
Note that $\Omega_t$ is a volume increasing (decreasing) perturbation if $\int_{\partial \Omega}  v \cdot \eta \; ~ds >0 (<0).$

We get strict inequality in the second part of the above theorem if at least one of the inequalities for $R,\beta,$ or $p$ is strict. The above results indicate that the convexity constraints in the currently available literature may be too restrictive for at least certain range of values of the Robin boundary parameter. In the next result we prove that when the boundary parameter is large enough,  domain monotonicity property holds for all $C^{2,\alpha}$ class domains.

\begin{theorem}\label{LargeBeta}
    Let $\Omega$ and $\Omega_t$ be as in Theorem \ref{Final formulae for lambda dot 0} and let $v$ be a $C^{2,\alpha}$ vector field such that $v(x)\cdot \eta(x)>0$ for all $x\in \partial\Omega$. Then there exists a constant $\beta^*(\Omega, v)>0$ such that  for $\beta>\beta^*$, $\dot{\lambda}_1(0)<0$.
\end{theorem}
The proof for the above theorem uses the convergence of the first Robin eigenvalue $\lambda_1$ and the corresponding eigenfunction to the first Dirichlet eigenvalue $\lambda_1^D$ and the eigenfunction corresponding
to it. 
There are works describing this asymptotics for the Laplace operator in the literature, see for example \cite{Ramm2009}, \cite{Filinovskiy2014}, and
\cite{Filinovskiy2017}.
Since  we could not find any reference for an appropriate convergence result for the Robin $p$-Laplacian,  we have included a  convergence result in Section 5.

In the next section, we present some results that would be required in the sequel. We prove Theorem \ref{Final formulae for lambda dot 0} in Section 3,  Theorem \ref{Ball case} in Section 4, and Theorem \ref{LargeBeta} in Section 5.

	\section{Preliminaries}
	\label{Preliminaries section}
In this section, we mostly follow the presentation in \cite{bandle2015second}. Let $\Omega \subset \mathbb{R}^n$ be a bounded domain in the H$\ddot{\text{o}}$lder class $C^{2,\alpha}.$ Let $P$ be a point on $\partial \Omega$ and $T_P$ denote the tangent space of $P.$ Then there is a neighborhood $\Omega_P$ of $P$ and a Cartesian coordinate system having an orthonormal basis $B_P:=\{e_i\}, i=1,\ldots,n, $ centered at $P,$ with $e_i \in T_P, i=1,\ldots,n-1, $ and $e_n$ in the direction of unit outer normal $\eta$. Let $(\xi_1, \xi_2,\ldots,\xi_n)$ denote the coordinates with respect to $B_P.$ Also, we make the assumption $ \Omega \cap \Omega_P =\{\xi \in \Omega_P : \xi_n <F(\xi_1, \xi_2,\ldots,\xi_{n-1}) \}, F \in C^{2,\alpha}. $ From this, it is clear that $F_{\xi_i}(0)=0$ for $i=1,2,\ldots,n-1$. Let $\xi ' := (\xi_1, \xi_2,\ldots,\xi_{n-1}).$ In the $\xi ' -$ coordinate system,  $x(\xi ') = (\xi_1, \xi_2,\ldots,\xi_{n-1}, F(\xi '))$ represents points on $\Omega_P \cap  \partial \Omega $ and $\tilde \eta  (\xi ')= (\tilde \eta_1,\tilde \eta_2,\ldots,\tilde \eta_n) $ the unit outer normal. We will follow the Einstein convention, where repeated indices indicate summation from $1$ to $n-1$ or from $1$ to $n$. For $i=1,2,\ldots,n-1,$ the vectors $x_{\xi_i}:=\frac{\partial x}{\partial \xi_i}$ span $T_P.$ Let $g_{ij}:=x_{\xi_i} \cdot x_{\xi_j}$ be the metric tensor of $ \partial \Omega,$ and its inverse be denoted by $g^{ij}.$

Note that any vector $v$ can be written as $v=v^\tau +(v \cdot \eta)\eta,$ where $v^\tau := g^{ij} (v \cdot x_{\xi_i}) x_{\xi_j} $ is the projection to the tangent space. Let $v:\partial \Omega \to \mathbb{R}^n$ be a smooth vector field. The tangential divergence of $v$ is defined as 
\begin{equation} 
\label{Equation to be used in sigma A}
\text{div}_{\partial \Omega} v=g^{ij} \tilde{v}_{\xi_i} \cdot x_{\xi_j}, 
\end{equation} where $ \tilde{v}=v(\xi',F(\xi')).$ For any $f \in C^1(\Omega_P),$ the tangential gradient of $f|_{\partial\Omega}$ is given by
\begin{equation}
    \label{grad tau f on boundary}
    \nabla ^\tau f =g^{ij} (\nabla f \cdot x_{\xi_i}) x_{\xi_j}.
\end{equation}

For $f\in C^1(\partial \Omega),$ $v\in C^{0,1}(\partial \Omega , \mathbb{R}^n ),$ and mean curvature $H$ of the boundary $\partial \Omega,$ the Gauss theorem on surfaces can be stated as 
\begin{equation}
\label{Gauss theorem on surfaces}
    \int_{\partial \Omega}f \text{ div}_{\;\partial \Omega} \;v~ds = - \int_{\partial \Omega}v \cdot \nabla ^\tau f ~ds + (n-1) \int_{\partial \Omega}f (v\cdot \eta) H ~ds.
\end{equation}

\subsection{Perturbation of domain}
\label{Perturbation of domain section}
Let $(\Omega_t)_t $ be a family of perturbations with $ \Omega_t :=\{y=x+tv(x)+o(t):x\in \Omega, |t| \text{ small}\}$ where $v=(v_1(x),v_2(x),..., v_n(x)) $ is a $C^{2,\alpha}$ vector field. The volume element of $\Omega_t$ and surface element of $\partial \Omega_t$ have been computed in \cite{bandle2015second}. We include the same here for the sake of completeness. 

\noindent Let $y(t,\Omega):=y.$ The Jacobian matrix (up to first order terms) of $y$ is given by $I +t D_v,$ where $(D_v)_{ij}=\frac{\partial v_i}{\partial x_j}.$ When $|t|$ is small, using Jacobi's formula, we get 
\begin{equation}
\label{J(t)}
J(t):=\text{det} (I+tD_v) =1+t \text{ div } v+ o(t).
\end{equation}
Now $|\Omega_t|=\int_\Omega J(t) ~dx = |\Omega|+ t \int_\Omega \text{ div } v ~dx +  o(t).$

 In order to compute the surface element of the boundary, $\partial \Omega_t $ can be represented locally as $ \{ y(\xi ') := x(\xi ')+t \tilde v(\xi ') : \xi ' \in \Omega_P \cap \{\xi_n=0\} \}.$ 
Let $a_{ij}:=x_{\xi_i} \cdot \tilde v(\xi_j) + x_{\xi_j} \cdot \tilde v(\xi_i), b_{ij}:=2\tilde v(\xi_i) \cdot \tilde v(\xi_j)$ and $g_{ij}^{t}:=g_{ij}+ta_{ij}+\frac{t^2}{2}b_{ij}.$
Then $|dy|^2=\left( g_{ij}+ta_{ij}+\frac{t^2}{2}b_{ij} \right) d\xi_i d\xi_j=g_{ij}^{t}d\xi_i d\xi_j.$
Now let $G:=(g_{ij}), G^{-1}:=(g^{ij}), A:=(a_{ij}), B:=(b_{ij}) $ and $ G^t:=(g_{ij}^{t}).$ The surface element on $\partial \Omega_t$ is given by $\sqrt{\text{det } G^t}d\xi '.$
Then $\sqrt{\text{det } G^t}=\left(\sqrt{\text{det } G}\right) k(x,t)^{1/2}, $ where $k(x,t):= \text{det }\left( I+tG^{-1}A+\frac{t^2}{2}G^{-1}B \right) .$ Let $\sigma_A:= \text{trace }G^{-1}A,  \sigma_B:= \text{trace }G^{-1}B$ and $ \sigma_{A^2}:= \text{trace }(G^{-1}A)^2.$
Using the Taylor expansion, for small $|t|,$ we have
$k(x,t)=1+t\sigma_A + o(t). $

Let $m(t):=\sqrt{k(x,t)}=1+\frac{t}{2}\sigma_A +o(t).$ Then the surface element $ds_t$ of $\partial \Omega_t$ can be written as $ds_t=m(t) ds,$ where $ds$ is the surface element of $\partial \Omega.$

 Using \eqref{Equation to be used in sigma A}, we have $\sigma_A=2 g^{ij} \tilde v_{\xi_j} \cdot x_{\xi_i} = 2 \text{div}_{\partial\Omega} v.$ It can be shown that $\frac{1}{2}\sigma_A= \text{div}_{\partial\Omega} v^\tau +(n-1)H(v \cdot \eta),$ where $H$ denotes the mean curvature of the boundary. So we have \begin{equation}
\label{m dot 0}
\dot m (0)=\frac{1}{2}\sigma_A= \text{div}_{\partial\Omega} v^\tau +(n-1)H(v \cdot \eta).
\end{equation}

\section{Shape derivative formula}
Consider the following energy functional 
\begin{equation}
\label{First energy functional}
E_{\Omega_t}(u):=\int_{\Omega_t} |\nabla_y u |^p dy -\int_{\Omega_t}\lambda(\Omega_t) |u|^p dy + \beta \int_{\partial \Omega_t}  |u|^p ~ds_t.
\end{equation}

\noindent Let $\tilde{u} \in W^{1,p}(\Omega_t)$ be a critical point of (\ref{First energy functional}), that is, $\tilde{u}$ satisfies the Euler-Lagrange equation 
	\begin{equation}
\label{Euler Lagrange equation in Omega_t}
\left\{
\begin{split}
\Delta_p \tilde{u}+\lambda(\Omega_t)  |\tilde{u}|^{p-2}\tilde{u}&=0\quad \rm{in} ~\Omega_t ,\\
|\nabla \tilde{u}|^{p-2}\frac{\partial \tilde{u}}{\partial\eta_t}+\beta|\tilde{u}|^{p-2}\tilde{u}&
=0\quad\rm{on}~\partial\Omega_t,
\end{split}
\right.
\end{equation} where $\eta_t$ denotes the outer unit normal to $\partial \Omega_t.$

\noindent Now set $E(t):=E_{\Omega_t}(\tilde{u})$ and $\lambda(t):=\lambda(\Omega_t).$ Consider $y=x+tv(x)+o(t) $ with $x(y)$ being its inverse. Using the change of variables, we get 
\begin{equation*}
\label{E in Omega}
E(t)=\int_{\Omega} \left[(\nabla \tilde{u}(t))^T A(t) \nabla \tilde{u}(t) \right]^\frac{p}{2} ~dx -\int_{\Omega}\lambda(t) |\tilde{u}(t)|^pJ(t)~dx + \beta \int_{\partial \Omega}  |\tilde{u}(t)|^p m(t)~ds.
\end{equation*}
where $A_{ij}(t):=\frac{\partial x_i}{\partial y_k}\frac{\partial x_j}{\partial y_k} \left(J(t)\right)^\frac{2}{p},$ where the Jacobian $J(t)$ is defined in \eqref{J(t)}, and $\tilde{u}(t):=\tilde{u}(x+tv(x),t), $ for $t \in (-\epsilon, \epsilon)$ with $\epsilon > 0$ small enough.
	
\noindent In $\Omega,$ $\tilde{u}(t)$ solves the corresponding Euler-Lagrange equation  
 \begin{equation}
 \label{Euler Lagrange equation in Omega}
 \left\{
 \begin{split}
 L_A\tilde{u}(t)+\lambda(t)  |\tilde{u}(t)|^{p-2}\tilde{u}(t)J(t)&=0\quad \rm{in} ~\Omega ,\\
 \partial_{\eta_A}\tilde{u}(t)+\beta|\tilde{u}(t)|^{p-2}\tilde{u}(t)m(t)&=0\quad\rm{on}~\partial\Omega,
 \end{split}
 \right.
 \end{equation}
 where $L_A(\tilde{u}(t)):=\frac{\partial}{\partial x_j}\{\left[(\nabla \tilde{u}(t))^T A \nabla \tilde{u}(t) \right]^\frac{p-2}{2}A_{ij} \tilde{u}_{x_i}(t) \}$ and \\ $\partial_ {\eta_A}(\tilde{u}(t)) :=\left[(\nabla \tilde{u}(t))^T A \nabla \tilde{u}(t) \right]^\frac{p-2}{2}A_{ij} \tilde{u}_{x_i}(t) \eta^j. $
 
 \noindent Writing (\ref{Euler Lagrange equation in Omega}) in the weak form, we get 
 \begin{equation}
 \label{Weak form of Euler Lagrange equation in Omega}
 \int_{\Omega} \left[(\nabla \tilde{u})^T A \nabla \tilde{u} \right]^\frac{p-2}{2} \left[(\nabla \tilde{u})^T A \nabla \phi \right]~dx + \beta \int_{\partial \Omega}  |\tilde{u}|^{p-2} \tilde{u} m \phi ~ds=\int_{\Omega}\lambda(t) |\tilde{u}|^{p-2}\tilde{u}J\phi ~dx, 
 \end{equation} 
for all $\phi \in W^{1,p}(\Omega).$

 Note that $\tilde{u}(t)$ can be expressed as $\tilde{u}(t)=\tilde{u}(0)+t\dot{\tilde{u}}(0)+o(t).$ Also $\tilde{u}(0):=u(x).$ $A_{ij}(t) $ can be expanded as $A_{ij}(t)=A_{ij}(0)+t\dot{A_{ij}}(0)+o(t).$

\begin{lemma}
	\label{Lemma for Aij}
	We have \begin{enumerate}
	    \item[(i)] $J(0)=1$ and $\dot{J}(0)=div\; v,$
	    \item[(ii)] $A_{ij}(0)= \delta_{ij}$ and $\dot{A_{ij}}(0)= \frac{2}{p}\delta_{ij}\;div\; v-\partial_jv_i-\partial_iv_j,$ where $\partial_jv_i:=\frac{\partial v_i}{\partial x_j}.$
	\end{enumerate}
\end{lemma}
\begin{proof}
	From (\ref{J(t)}), it follows that $J(0)=1$ and $\dot{J}(0)=div\; v.$
	
	\noindent For small positive $t,$ 
	$$\frac{\partial x_i}{\partial y_k}=(I+tD_v)^{-1}_{ik} =(I-tD_v+o(t))_{ik}=\delta_{ik}-t\frac{\partial v_k}{\partial x_i}+o(t).$$
	So we have 
	$$A_{ij}(0)= \delta_{ik}\delta_{jk}=\delta_{ij}.$$
	$$\begin{aligned}
	\dot{A_{ij}}(0)&=\left[ \frac{d}{dt}\left(\frac{\partial x_i}{\partial y_k} \right)\frac{\partial x_j}{\partial y_k} J(t)^\frac{2}{p} + \frac{\partial x_i}{\partial y_k}\frac{d}{dt}\left(\frac{\partial x_j}{\partial y_k}  \right) J(t)^\frac{2}{p} + \frac{2}{p} \frac{\partial x_i}{\partial y_k}\frac{\partial x_j}{\partial y_k} J(t)^{\frac{2}{p}-1} \dot{J}(t) \right]_{t=0} \\
	&=-\partial_jv_i-\partial_iv_j+\frac{2}{p}\delta_{ij}\;div\; v.
	\end{aligned} $$
\end{proof} 
Recall that the first eigenvalue $\lambda_1(t):=\lambda_1(\Omega_t)$ of (\ref{Euler Lagrange equation in Omega_t}) is given by the Rayleigh quotient $$\lambda_1(t)=\frac{\int_{\Omega_t} |\nabla_y \tilde{u} |^p dy + \beta \int_{\partial \Omega_t}  |\tilde{u}|^p ~ds_t}  {\int_{\Omega_t}  |\tilde{u}|^p dy}. $$
After a change of variables, we have 
\begin{equation}
\label{Rayleigh quotient after change of variables}
\lambda_1(t)=\frac{\int_{\Omega}  \left[(\nabla \tilde{u}(t))^T A(t) \nabla \tilde{u}(t) \right]^\frac{p}{2} ~dx + \beta \int_{\partial \Omega}  |\tilde{u}(t)|^p m(t) ~ds}  {\int_{\Omega}  |\tilde{u}(t)|^p J(t) ~dx}. 
\end{equation} 

\noindent As \eqref{Euler Lagrange equation in Omega_t} is $(p-1)$ homogeneous, we impose the normalization $\int_{\Omega}  |\tilde{u}(t)|^p J(t) ~dx=1.$ Then  
\begin{equation}
\label{Derivative of normalization}
\frac{d}{dt}\left( \int_{\Omega}  |\tilde{u}(t)|^p J(t) ~dx \right) = p\int_{\Omega}  |\tilde{u}(t)|^{p-2} \tilde{u}(t) \dot{\tilde{u}}(t)J(t) ~dx +\int_{\Omega}  |\tilde{u}(t)|^p \dot{J}(t) ~dx=0.
\end{equation}

\noindent Choosing $\dot{\tilde{u}} $ as a test function in \eqref{Weak form of Euler Lagrange equation in Omega}, we get 
\begin{equation}
\label{After putting test function in weak form}
\int_{\Omega} \left[(\nabla \tilde{u})^T A \nabla \tilde{u} \right]^\frac{p-2}{2} \left[(\nabla \tilde{u})^T A \nabla \dot{\tilde{u}} \right]~dx + \beta \int_{\partial \Omega}  |\tilde{u}|^{p-2} \tilde{u} \dot{\tilde{u}} m  ~ds=\int_{\Omega}\lambda_1(t) |\tilde{u}|^{p-2}\tilde{u} \dot{\tilde{u}} J ~dx.
\end{equation}

\noindent Multiplying both sides of (\ref{After putting test function in weak form}) by $p$ and substituting from (\ref{Derivative of normalization}), we get
\begin{equation}
\label{For substituting in lambda dot}
p\int_{\Omega} \left[(\nabla \tilde{u})^T A \nabla \tilde{u} \right]^\frac{p-2}{2} \left[(\nabla \tilde{u})^T A \nabla \dot{\tilde{u}} \right]~dx + p\beta \int_{\partial \Omega}  |\tilde{u}|^{p-2} \tilde{u} \dot{\tilde{u}} m  ~ds=-\lambda_1(t)\int_{\Omega}  |\tilde{u}|^p \dot{J}~dx.
\end{equation}
Differentiating (\ref{Rayleigh quotient after change of variables}) with respect to $t$ on both sides and using (\ref{For substituting in lambda dot}), we have 
\begin{equation*}
\label{lambda dot t}
\dot{\lambda}_1(t)=\frac{p}{2}\int_{\Omega} \left[(\nabla \tilde{u})^T A \nabla \tilde{u} \right]^\frac{p-2}{2} \left[(\nabla \tilde{u})^T \dot{A} \nabla {\tilde{u}} \right]~dx -\lambda_1(t)\int_{\Omega}  |\tilde{u}|^p \dot{J}~dx+ \beta \int_{\partial \Omega}  |\tilde{u}|^p \dot{m} ~ds.
\end{equation*} 

\noindent So 
\begin{equation*}
\label{lambda dot 0 starting equation}
\dot{\lambda}_1(0)=\frac{p}{2}\int_{\Omega} \left[(\nabla {u})^T A(0) \nabla {u} \right]^\frac{p-2}{2} \left[(\nabla {u})^T \dot{A}(0) \nabla {u} \right]~dx -\lambda_1(0)\int_{\Omega}  |u|^p \dot{J}(0)~dx+ \beta \int_{\partial \Omega}  |u|^p \dot{m}(0) ~ds.
\end{equation*} 
Let $\dot{\Lambda}_1:=\frac{p}{2}\int_{\Omega} \left[(\nabla {u})^T A(0) \nabla {u} \right]^\frac{p-2}{2} \left[(\nabla {u})^T \dot{A}(0) \nabla {u} \right]~dx, \; \dot{\Lambda}_2:=-\lambda_1(0)\int_{\Omega}  |u|^p \dot{J}(0)~dx$ and $\dot{\Lambda}_3:=\beta \int_{\partial \Omega}  |u|^p \dot{m}(0) ~ds.$

\noindent Now $$ \dot{\Lambda}_1= \frac{p}{2}\int_{\Omega} \left[u_{x_i} A_{ij}(0) u_{x_j} \right]^\frac{p-2}{2} \left[u_{x_i} \dot{A}_{ij}(0) u_{x_j} \right] ~dx. $$

\noindent Using Lemma \ref{Lemma for Aij} and simplifying,
\begin{equation}
\label{la_1 dash}
   \dot{\Lambda}_1=\int_{\Omega} |\nabla {u}|^{p} \text{ div } v ~dx -p\int_{\Omega} |\nabla {u}|^{p-2}\partial_j v_i u_{x_i} u_{x_j} ~dx.
\end{equation}

If the eigenfunction $u$ had $C^2$ regularity, we could have used integration by parts in the above equation to obtain a simplified expression. However, since $u$ has only $C^{1,\alpha}$ regularity, following the approach in \cite{melian2001perturbation} and \cite{del2009some}, for $\varepsilon>0$ small we consider the  perturbed problem:  

\begin{equation}
	\label{Epsilon Robin p-Laplacian}
	\left\{
	\begin{split}
\mbox{div} ((|\nabla u_\varepsilon|^2+\varepsilon ^2)^\frac{p-2}{2} \nabla u_\varepsilon)+\lambda_\varepsilon  |u_\varepsilon|^{p-2}u_\varepsilon&=0\quad \rm{in} ~\Omega,\\
	(|\nabla u_\varepsilon|^2+\varepsilon ^2)^\frac{p-2}{2}\frac{\partial u_\varepsilon}{\partial\eta}+\beta|u_\varepsilon|^{p-2}u_\varepsilon&=0\quad\rm{on}~\partial\Omega.
	\end{split}
	\right.
	\end{equation} 
Then $u_\varepsilon$ is in  $C^{2,\theta}(\overline{\Omega})$ for some $\theta \in (0,1) $  and $u_\varepsilon \to u$ in $C^1(\overline{\Omega})$ as $\varepsilon \to 0$ (see \cite{MR0241822} and \cite{lieberman1988boundary}).

\noindent Consider 
\begin{equation}
\label{Lambda_1 epsilon dot}
   \dot{\Lambda}_1^\varepsilon=\int_{\Omega} (|\nabla u_\varepsilon|^2+\varepsilon ^2)^\frac{p}{2} \text{ div } v ~dx -p\int_{\Omega}(|\nabla u_\varepsilon|^2+\varepsilon ^2)^\frac{p-2}{2}\partial_j v_i (u_\varepsilon)_{x_i} (u_\varepsilon)_{x_j} ~dx.
\end{equation}

\noindent Note that 
\begin{equation}
    \label{p-Laplace expansion for perturbed problem}
    \mbox{div} ((|\nabla u_\varepsilon|^2+\varepsilon ^2)^\frac{p-2}{2} \nabla u_\varepsilon)=(|\nabla u_\varepsilon|^2+\varepsilon ^2)^\frac{p-2}{2} \Delta u_\varepsilon + (p-2) (|\nabla u_\varepsilon|^2+\varepsilon ^2)^\frac{p-4}{2} (u_\varepsilon)_{x_i} (u_\varepsilon)_{x_j x_i} (u_\varepsilon)_{x_j}.
\end{equation}
\noindent After applying the integration by parts in both the integrals of \eqref{Lambda_1 epsilon dot} and using \eqref{p-Laplace expansion for perturbed problem}, we get 
 
\begin{equation}
 \label{Lambda_1 epsilon dot new}
    \begin{aligned}
     \dot{\Lambda}_1^\varepsilon &= p\int_{\Omega} \mbox{div} ((|\nabla u_\varepsilon|^2+\varepsilon ^2)^\frac{p-2}{2} \nabla u_\varepsilon) (\nabla u_\varepsilon \cdot v)~dx \\
     &-p\int_{\partial \Omega}(|\nabla u_\varepsilon|^2+\varepsilon ^2)^\frac{p-2}{2}(\nabla u_\varepsilon \cdot v)(\nabla u_\varepsilon \cdot \eta) ~ds+\int_{\partial \Omega} (|\nabla u_\varepsilon|^2+\varepsilon ^2)^\frac{p}{2}(v \cdot \eta)~ds.
   \end{aligned}
 \end{equation}

\noindent Using the integration by parts in $\dot{\Lambda}_2,$ we have  $$\dot{\Lambda}_2=  \lambda_1 (0) \int_{\Omega} p |u|^{p-2} u (\nabla u \cdot v) ~dx - \lambda_1 (0)  \int_{\partial \Omega} |u|^p (v \cdot \eta) ~ds. $$
\noindent Let $$\dot{\Lambda}_2^\varepsilon=  \lambda_\varepsilon \int_{\Omega} p |u_\varepsilon|^{p-2} u_\varepsilon (\nabla u_\varepsilon \cdot v) ~dx - \lambda_1 (0)  \int_{\partial \Omega} |u|^p (v \cdot \eta) ~ds.$$
Since $\lambda_\varepsilon |u_\varepsilon|^{p-2} u_\varepsilon = - \mbox{div} ((|\nabla u_\varepsilon|^2+\varepsilon ^2)^\frac{p-2}{2} \nabla u_\varepsilon)$ in $\Omega ,$  $\dot{\Lambda}_2^\varepsilon$ becomes 
\begin{equation}
 \label{Lambda_2 epsilon dot}
    \dot{\Lambda}_2^\varepsilon = - p\int_{\Omega} \mbox{div} ((|\nabla u_\varepsilon|^2+\varepsilon ^2)^\frac{p-2}{2} \nabla u_\varepsilon) (\nabla u_\varepsilon \cdot v)~dx-\lambda_1 (0)  \int_{\partial \Omega} |u|^p (v \cdot \eta) ~ds.
\end{equation}

\noindent Now consider $\dot{\Lambda}_3. $ Using (\ref{m dot 0}), we get
\begin{equation}
  \label{Lambda_3 dot}
 \dot{\Lambda}_3= \beta \int_{\partial \Omega} |u|^p \{ \text{div}_{\partial\Omega} v^\tau +(n-1)H(v \cdot \eta) \} ~ds.  
\end{equation}

\noindent Note that $\dot{\Lambda}_1^\varepsilon$ and $\dot{\Lambda}_2^\varepsilon$ converges to $\dot{\Lambda}_1$ and $\dot{\Lambda}_2,$ respectively as $\varepsilon \to 0.$ Hence  
\begin{equation}
\label{First formula for lambda dot 0}
\begin{aligned}
\dot{\lambda}_1(0)&=\displaystyle{\lim_{\varepsilon \to 0}}\;\dot{\Lambda}_1^\varepsilon+\dot{\Lambda}_2^\varepsilon+\dot{\Lambda}_3=\dot{\Lambda}_1+\dot{\Lambda}_2+\dot{\Lambda}_3\\
&=\int_{\partial \Omega} \{\left[|\nabla u|^p - \lambda_1 (0) |u|^p + \beta |u|^p (n-1)H \right] (v \cdot \eta) \\
&\quad \quad -p |\nabla u|^{p-2}(\nabla u \cdot v)(\nabla u \cdot \eta) + \beta |u|^p  \text{div}_{\partial\Omega} v^\tau  \} ~ds. 
\end{aligned} 
\end{equation}
\noindent Using the boundary condition of (\ref{Euler Lagrange equation in Omega_t}) in (\ref{First formula for lambda dot 0}),
\begin{equation}
\label{Third formula for lambda dot 0}
\begin{aligned}
\dot{\lambda}_1(0)&=\int_{\partial \Omega} \{\left[|\nabla u|^p - \lambda_1 (0) |u|^p + \beta |u|^p (n-1)H \right] (v \cdot \eta) \\
&\quad \quad +p \beta |u|^{p-2} u (\nabla u \cdot v) + \beta |u|^p  \text{div}_{\partial\Omega} v^\tau  \} ~ds. 
\end{aligned} 
\end{equation}

\noindent We can write 
$$ \begin{aligned}
|u|^{p-2} (\nabla u \cdot v )&= \left( \left( |u|^{p-2} \nabla u  \right)^\tau +  \left( |u|^{p-2} \nabla u \cdot \eta \right) \eta \right) \cdot \left( v^\tau + (v \cdot \eta) \eta \right) \\
&= \left( |u|^{p-2} \nabla u  \right)^\tau \cdot v^\tau + \left( |u|^{p-2} \nabla u \cdot \eta \right) (v \cdot \eta).
\end{aligned} $$
Substituting this expression in \eqref{Third formula for lambda dot 0}, we arrive at 
\begin{equation}
\label{Fourth formula for lambda dot 0}
\begin{aligned}
\dot{\lambda}_1(0)&=\int_{\partial \Omega} \{\left[|\nabla u|^p - \lambda_1 (0) |u|^p + \beta |u|^p (n-1)H + p \beta u \left( |u|^{p-2} \nabla u \cdot \eta \right) \right] (v \cdot \eta) \\
&\quad \quad +p \beta u \left( |u|^{p-2} \nabla u  \right)^\tau \cdot v^\tau + \beta |u|^p  \text{div}_{\partial\Omega} v^\tau  \} ~ds. 
\end{aligned} 
\end{equation}
Using the Gauss theorem on surfaces (\ref{Gauss theorem on surfaces}), 
	 \begin{equation*}
	 \begin{aligned}
	 \int_{\partial \Omega}  \beta |u|^p  \text{div}_{\partial\Omega} v^\tau ~ds &=- \beta\int_{\partial \Omega}v^\tau \cdot \nabla ^\tau |u|^p ~ds + \beta(n-1) \int_{\partial \Omega}|u|^p (v^\tau\cdot\eta) H ~ds \\
	 &=- \beta\int_{\partial \Omega}v^\tau \cdot \nabla ^\tau |u|^p ~ds.
	 \end{aligned}
	 \end{equation*}
	By (\ref{grad tau f on boundary}), $\nabla ^\tau |u|^p$ can be rewritten as follows,
	\begin{equation*}
	 \nabla ^\tau |u|^p =g^{ij} (\nabla |u|^p \cdot x_{\xi_i}) x_{\xi_j} =g^{ij}(p|u|^{p-2}u\nabla u \cdot x_{\xi_i}) x_{\xi_j} =pu(|u|^{p-2}\nabla u)^\tau.
	 \end{equation*}
	 So \begin{equation}
	 \label{For cancelling the extra terms in lambda dot 0}
	    \int_{\partial \Omega}  \beta |u|^p  \text{div}_{\partial\Omega} v^\tau ~ds = - \int_{\partial \Omega}p \beta u v^\tau \cdot  (|u|^{p-2}\nabla u)^\tau ~ds.
	 \end{equation}
	 Using (\ref{For cancelling the extra terms in lambda dot 0}), the expression in (\ref{Fourth formula for lambda dot 0}) reduces to 
	 \begin{equation*}
	 \label{Final formula for lambda dot 0 -1}
	     \dot{\lambda}_1(0)=\int_{\partial \Omega} \{\left[|\nabla u|^p - \lambda_1 (0) |u|^p + \beta |u|^p (n-1)H + p \beta u \left( |u|^{p-2} \nabla u \cdot \eta \right) \right] (v \cdot \eta)  \} ~ds.
	 \end{equation*}
Using the boundary condition of (\ref{Euler Lagrange equation in Omega_t}), we can rewrite the above equation as 
\begin{equation*}
\label{Final formula for lambda dot 0 -2}
 \dot{\lambda}_1(0)=\int_{\partial \Omega}  \{ \left[|\nabla u|^p - \lambda_1 (0) |u|^p + \beta |u|^p (n-1)H - p \beta^2 \frac{|u|^{2p-2}}{|\nabla u|^{p-2}}  \right] (v \cdot \eta) ~ds. 
\end{equation*}
Thus we have derived the shape derivative results summarized in Theorem \ref{Final formulae for lambda dot 0}.

\begin{corollary}
\label{lambda dot 0=0}
	     For a family $\Omega_t$ of perturbations of $\Omega$ with the property that $\int_{\Omega} \text{\rm{div }}v \;~dx =0,$ $\dot{\lambda}_1(0)=0$ if and only if, on $\partial \Omega,$
	     \begin{equation*}
	         |\nabla u|^p - \lambda_1 (0) |u|^p + \beta |u|^p (n-1)H - p \beta^2 \frac{|u|^{2p-2}}{|\nabla u|^{p-2}}= \rm{constant.}
	     \end{equation*}
	 \end{corollary}

\begin{proof}
    Let $\zeta (x) := |\nabla u|^p - \lambda_1 (0) |u|^p + \beta |u|^p (n-1)H - p \beta^2 \frac{|u|^{2p-2}}{|\nabla u|^{p-2}}.$ By Green's theorem, $\int_{\partial \Omega}(v\cdot \eta)~ds=\int_{\Omega} \text{div } v ~dx=0 $ as $\Omega_t$ is volume preserving. If $\zeta $ is a constant then $\dot{\lambda}_1(0)= \int_{\partial \Omega}\zeta (v\cdot \eta)~ds= 0. $\\
    Define $\overline{\zeta}(x):=\frac{1}{|\partial \Omega|}\int_{\partial \Omega}\zeta ~ds.$ So $\int_{\partial \Omega}(v\cdot \eta)\overline{\zeta}~ds=\overline{\zeta}\int_{\partial \Omega}(v\cdot \eta)~ds=0.$ \\
    Assume $\dot{\lambda}_1(0)=0.$ This implies $\int_{\partial \Omega}\zeta (v\cdot \eta)~ds=0.$ So 
    \begin{equation}
    \label{In between step}
        \int_{\partial \Omega}\zeta (v\cdot \eta)~ds=\int_{\partial \Omega}(\zeta-\overline{\zeta}) (v\cdot \eta)~ds.
    \end{equation}
    Now define $\zeta^\pm =\text{ max} \{0, \pm (\zeta-\overline{\zeta})\}. $ If $\zeta$ is not a constant, $\zeta^\pm \neq0.$ Therefore, a volume preserving perturbation can be constructed so that $(v\cdot \eta)>0$ in supp $\zeta ^+$ and $(v\cdot \eta)<0$ in supp $\zeta ^-.$ Then, using (\ref{In between step}), $$\dot{\lambda}_1(0)= \int_{\partial \Omega}(\zeta-\overline{\zeta}) (v\cdot \eta)~ds = \int_{\text{supp } \zeta ^+}\zeta^+ (v\cdot \eta)~ds - \int_{\text{supp } \zeta ^-}\zeta^- (v\cdot \eta)~ds > 0,$$ which is a contradiction. 
\end{proof}
Note that such a family of perturbations will be volume preserving of the first order. 
\section{Domain monotonicity results for perturbations of a ball}
Let $\Omega$ be a ball of radius $R$ in $\mathbb{R}^n$ centered at the origin. Then $\eta=\frac{x}{R}$ and $H=\frac{1}{R}.$ Let  $u$ be an eigenfunction corresponding to $\lambda_1(\Omega)$.
Note that the functional $J(z)$ in the characterization \eqref{Rayleigh}  of $\lambda_1(\Omega)$ is rotation invariant. Hence any rotation of $u$ about the origin, say $\hat{u}$, will also be a minimizer of $J(z)$.  As $\lambda_1(\Omega)$ is simple, it follows that the eigenfunction $u$ is radial. Let $u_r=\frac{du}{dr}$ with $r=|x|.$ Now $\nabla u \cdot \eta =u_{x_i}\frac{x_i}{R}=u_r\frac{\partial r}{\partial x_i}\frac{x_i}{R}=u_r\frac{r}{R}.$ On $\partial \Omega,$ $\nabla u \cdot \eta =u_r(R),$ and hence $|\nabla u \cdot \eta |=|u_r(R)|.$ We have $|\nabla u |^2=\displaystyle{\sum_{i}}\left (  u_r\frac{\partial r}{\partial x_i}\right )^2=u_r^2\displaystyle{\sum_{i}}\left (\frac{x_i}{r} \right )^2=u_r^2.$ So on the boundary, we have $|\nabla u \cdot \eta |= |\nabla u |.$ Using the boundary condition, we get $|\nabla u|=\beta^\frac{1}{p-1}|u|.$

Now the above computations can be used to prove Theorem \ref{Ball case}. 
\begin{proof}[Proof of Theorem \ref{Ball case}]
 \begin{enumerate}
 \item [(i)] We have 
 $$\begin{aligned} 
    \dot{\lambda}_1(0)&=\int_{\partial \Omega}  \left \{ \left[|\nabla u|^p - \lambda_1 (0) |u|^p + \beta |u|^p \frac{(n-1)}{R} - p \beta^2 \frac{|u|^{2p-2}}{|\nabla u|^{p-2}}  \right] (v \cdot \frac{x}{R}) \right \}  ~ds \\
    &=\int_{\partial \Omega} \left \{ \left[\beta^\frac{p}{p-1}|u|^{p} - \lambda_1 (0) |u|^p + \beta  \frac{(n-1)}{R}|u|^p - p \beta^2 \frac{|u|^{2p-2}}{\beta^\frac{p-2}{p-1}|u|^{p-2}}  \right] (v \cdot \frac{x}{R}) \right \}  ~ds \\
    &=\int_{\partial \Omega}  \left \{ \left[\beta^\frac{p}{p-1}(1-p) - \lambda_1 (0)  + \beta  \frac{(n-1)}{R}  \right]|u|^{p} (v \cdot \frac{x}{R}) \right \} ~ds\\
    &<0 (>0) \text{ when } \beta > \left(\frac{n-1}{R(p-1)}\right)^{p-1} \text{ since } u \text{ is radial and }\int_{\partial \Omega} v \cdot \eta ~ds>0 \; (<0) \text{ on } \partial \Omega.
    \end{aligned}
    $$
    \item [(ii)] Since $R \geq 1,\beta \geq 1$ and $p \geq n,$
    $$
\begin{aligned} 
	\dot{\lambda}_1(0)&=\int_{\partial \Omega}  \left \{ \left[\beta^\frac{p}{p-1}(1-p) - \lambda_1 (0)  + \beta  \frac{(n-1)}{R}  \right]|u|^{p} (v \cdot \frac{x}{R}) \right \} ~ds \\
&\leq \int_{\partial \Omega}  \left \{ \beta \left[(1-p) - \frac{\lambda_1 (0) }{\beta} +  (n-1)  \right]|u|^{p} (v \cdot \frac{x}{R}) \right \} ~ds \\
&=\int_{\partial \Omega}  \left \{ \beta \left[(n-p) - \frac{\lambda_1 (0) }{\beta}   \right]|u|^{p} (v \cdot \frac{x}{R}) \right \} ~ds \\
&\leq 0 (\geq 0) \text{ since } u \text{ is radial and }\int_{\partial \Omega} v \cdot \eta ~ds >0 \; (<0) \text{ on } \partial \Omega.
\end{aligned}
$$ 
\end{enumerate}
\end{proof}
From (i) it follows that for a fixed $\beta >0,$ if $R>\frac{ n-1}{\beta^\frac{1}{p-1} (p-1)},$ then $\dot{\lambda}_1(0)<0 \; (>0)$ for a smooth perturbation $v$ of the domain with $\int_{\partial\Omega}v \cdot \eta ~ds >0 \; (<0)$ on $\partial \Omega$.


\section{Domain monotonicity for large $\beta$}
Let us denote the first eigenvalue of 
\eqref{Robin p-Laplacian} by $\lambda_1^\beta$ and
$\phi_\beta$ be the corresponding normalized eigenfunction with $\|\phi_\beta\|_{\infty} = 1$. Note that $\phi_\beta \in C^{1,\alpha}(\overline{\Omega})$ by \cite[Theorem 2]{lieberman1988boundary}.
Let us also introduce the first eigenvalue of the  $p$-Laplace
operator with Dirichlet boundary condition
$$
\lambda_1^D
=
\inf\left\{
\frac{\int_{\Omega} |\nabla u|^p \,dx}{\int_{\Omega} |u|^p \,dx}:~
u \in W_0^{1,p}(\Omega) \setminus\{0\}
\right\}
$$
and denote by $\phi_D$ the corresponding normalized ($\|\phi_D\|_{\infty} = 1$) first eigenfunction.

It is known that $\lambda_1^\beta$  and $\phi_\beta$ converge to $\lambda_1^D$  and $\phi_D$  respectively as  the boundary parameter $\beta \to \infty.$ Since we could not find an appropriate reference for this convergence, we prove the following:   

\begin{proposition}
\label{Robin limit goes to dirichlet}
	We have $\lambda_1^\beta \to \lambda_1^D$ and $\phi_\beta \to \phi_D$ in $C^1(\overline{\Omega})$ as $\beta \to \infty$.
\end{proposition}
\begin{proof}
	First of all, we have $0<\lambda_1^\beta \leq \lambda_1^D$ for any $\beta>0$. 
	(The upper bound follows by substituting $u_D$ into the definition of $\lambda_1^\beta$.)
	Therefore, the monotonic sequence $\{\lambda_1^\beta\}$ converges to some $\lambda_0 \in [0,\lambda_1^D]$.
	
	On the other hand, since we require $\|\phi_\beta\|_{L^\infty} = 1$ for any $\beta>0$, $\lambda_1^\beta |\phi_\beta|^{p-2}\phi_\beta$
	 is uniformly bounded in $L^\infty(\Omega)$, and hence 
	we can apply \cite[Theorem 2]{lieberman1988boundary} to conclude that  $\{\|\phi_\beta\|_{C^{1,\alpha}(\overline{\Omega})}\}$ is bounded.
	Therefore, by the Arzela-Ascoli theorem, there exists some $\phi \in C^1(\overline{\Omega})$ such that $\phi_\beta \to \phi$ in $C^1(\overline{\Omega})$.
	This should also imply that $\phi_\beta \to \phi$ in $W^{1,p}(\Omega)$.
	
	Next we show that $\phi=0$ on $\partial \Omega$.
	Namely, let us show that $\int_{\partial \Omega} |\phi|^p \,ds = 0$.
	Suppose, by contradiction, that $\int_{\partial \Omega} |\phi|^p \,ds > 0$.
	In view of the convergence, there exists $C>0$ such that $\int_{\partial \Omega} |\phi_\beta|^p \,ds > C > 0$ for all sufficiently large $\beta>0$.
	Recalling that $\|\phi_\beta\|_{L^\infty} = 1$, we also have $\int_{\Omega} |\phi_\beta|^p \,dx \leq C_1$ for some $C_1>0$ and all $\beta>0$.
	However, this shows that 
	$$
	\lambda_1^\beta
	=
	\frac{\int_{\Omega} |\nabla \phi_\beta|^p \,dx + \beta \int_{\partial \Omega} |\phi_\beta|^p \,ds}{\int_{\Omega} |\phi_\beta|^p \,dx}
	\to 
	\infty,
	$$
	which is impossible.
	Therefore, $\int_{\partial \Omega} |\phi|^p \,ds = 0$ and hence $\phi=0$ on $\partial \Omega$.
	
	Note that $\phi \not\equiv 0$ in $\Omega$. This follows from the fact that $\|\phi_\beta\|_{\infty} = 1$ and $\phi_\beta \to \phi$ in $C^1(\overline{\Omega}).$ 
	So we can conclude that $\phi \in W_0^{1,p}(\Omega) \setminus \{0\}$, and by the convergence $\phi_\beta \to \phi$ in $W^{1,p}(\Omega)$ we get
	$$
	\lim_{\beta \to \infty} \lambda_1^\beta
	=
	\frac{\int_{\Omega} |\nabla \phi|^p \,dx 
		+ 
	\lim\limits_{\beta \to \infty} \left(\beta \int_{\partial \Omega} |\phi_\beta|^p \,ds\right)}{\int_{\Omega} |\phi|^p \,dx}
\leq \lambda_1^D.
	$$
	On the other hand, 
	$$
		\lim\limits_{\beta \to \infty} \left(\beta \int_{\partial \Omega} |\phi_\beta|^p \,ds\right) \geq 0,
	$$
	and hence we also have
	$$
	\lambda_1^D \leq 
	\frac{\int_{\Omega} |\nabla \phi|^p \,dx}{\int_{\Omega} |\phi|^p \,dx}
	\leq
	\liminf_{\beta \to \infty} \lambda_1^\beta.
	$$
	Thus, we conclude that $\lim\limits_{\beta \to \infty} \lambda_1^\beta = \lambda_1^D$. 
	Hence, $\phi \equiv \phi_D$, up to a multiplication.
\end{proof}
Let $u_\beta$ be the first eigenfunction of \eqref{Robin p-Laplacian} with the normalization $\int_{\Omega}  |u_\beta|^p ~dx=1$ and $u_D$ be the first eigenfunction of the Dirichlet $p$-Laplace eigenvalue problem normalized as $\int_{\Omega}  |u_D|^p ~dx=1.$ It can be shown that the convergence ($u_\beta \to u_D$ in $C^1(\overline{\Omega})$ as $\beta \to \infty$) holds true for these normalized eigenfunctions as well.

We use this convergence result to prove the persistence of  domain monotonicity  for large $\beta$.

\noindent{\bf Proof of Theorem \ref{LargeBeta}:}
Assume $ v \cdot \eta > 0.$
We have \begin{equation*}
\begin{aligned}
	     \dot{\lambda}_1^\beta(0)&=\int_{\partial \Omega} \{\left[|\nabla u_\beta|^p - \lambda_1 (0) |u_\beta|^p + \beta |u_\beta|^p (n-1)H + p \beta u_\beta \left( |u_\beta|^{p-2} \nabla u_\beta \cdot \eta \right) \right] (v \cdot \eta)  \} ~ds\\
	     &=\int_{\partial \Omega} \{\left[|\nabla u_\beta|^p - \lambda_1 (0) |u_\beta|^p + \beta |u_\beta|^p (n-1)H -p \beta  |u_\beta|^{p-1} |\nabla u_\beta \cdot \eta| \right] (v \cdot \eta)  \} ~ds \; \text{ since } \nabla u_\beta \cdot \eta < 0. 
\end{aligned}	     
	 \end{equation*}
	 From the boundary condition, we have $|u_\beta|^{p-1}=\frac{1}{\beta}|\nabla u_\beta|^{p-2}|\nabla u_\beta \cdot \eta|.$ So the previous equation can be written as
	 \begin{equation*}
	  \dot{\lambda}_1^\beta(0)=\int_{\partial \Omega} \{\left[|\nabla u_\beta|^p - \lambda_1 (0) |u_\beta|^p + \beta |u_\beta|^p (n-1)H - p |\nabla u_\beta|^{p-2}|\nabla u_\beta \cdot \eta|^2 \right] (v \cdot \eta)  \} ~ds. 
	 \end{equation*}
Since $\lambda_1^\beta \to \lambda_1^D,$ and $H$ is bounded, given $\varepsilon > 0,$ there exists $\tilde{\beta}(\Omega)$ such that 
\begin{equation}\label{beta u^p going to zero}
    \left |\int_{\partial \Omega}\beta |u_\beta|^p (n-1)H(v \cdot \eta) ~ds  \right | < \varepsilon \text{  for } \beta > \tilde{\beta}.
\end{equation}
So for $\beta > \tilde{\beta},$ \begin{equation*}
    \dot{\lambda}_1^\beta(0) < \int_{\partial \Omega} \{\left[|\nabla u_\beta|^p   - p |\nabla u_\beta|^{p-2}|\nabla u_\beta \cdot \eta|^2 \right] (v \cdot \eta)  \} ~ds + \varepsilon. 
\end{equation*}
As $u_\beta \to u_D $ in $C^1(\overline{\Omega}),$ there exists $\hat{\beta}$ such that for all $\beta > \hat{\beta},$	 
\begin{equation*}
    \begin{aligned}
    |\nabla u_\beta|^p   - p |\nabla u_\beta|^{p-2}|\nabla u_\beta \cdot \eta|^2 & < (|\nabla u_D|+\varepsilon)^p - p (|\nabla u_D|-\varepsilon)^{p-2}(|\nabla u_D \cdot \eta|-\varepsilon)^2 \\
    & =(|\nabla u_D|+\varepsilon)^p - p (|\nabla u_D|-\varepsilon)^{p}  \text{\; as } \nabla u_D \cdot \eta = \nabla u_D. 
    \end{aligned}	  
\end{equation*}
Thus for $\beta >\beta^*:= max\{\tilde{\beta},\hat{\beta}\}$ we have,
 $$\dot{\lambda}_1^\beta(0) < \int_{\partial \Omega} \{\left[(1-p)|\nabla u_D|^p +o(\varepsilon) \right] (v \cdot \eta)  \} ~ds+\varepsilon. $$
 Hence for $\epsilon$ sufficiently small there exists a $\beta^*>0$ such that for $\beta >\beta^*$
 $$\dot{\lambda}_1^\beta(0) < 0.$$

\qed

	\nocite{henrot2017shape}
	\nocite{giorgi2005monotonicity}
	\bibliographystyle{amsplain}
	\bibliography{reference}

\providecommand{\bysame}{\leavevmode\hbox to3em{\hrulefill}\thinspace}
\providecommand{\MR}{\relax\ifhmode\unskip\space\fi MR }
\providecommand{\MRhref}[2]{%
  \href{http://www.ams.org/mathscinet-getitem?mr=#1}{#2}
}
\providecommand{\href}[2]{#2}
\begin{thebibliography}{10}

\bibitem{Ramm2009}
H.~D. Alber and A.~G. Ramm, \emph{Asymptotics of the solution to {R}obin
  problem}, J. Math. Anal. Appl. \textbf{349} (2009), no.~1, 156--164.
  \MR{2455738}

\bibitem{anoop2018strict}
T~V Anoop, Vladimir Bobkov, and Sarath Sasi, \emph{On the strict monotonicity
  of the first eigenvalue of the $p$-{L}aplacian on annuli}, Transactions of
  the American Mathematical Society \textbf{370} (2018), no.~10, 7181--7199.

\bibitem{anoop2016structure}
T~V Anoop, P~Dr{\'a}bek, and Sarath Sasi, \emph{On the structure of the second
  eigenfunctions of the $p$-{L}aplacian on a ball}, Proceedings of the American
  Mathematical Society \textbf{144} (2016), no.~6, 2503--2512.

\bibitem{bandle2015second}
Catherine Bandle and Alfred Wagner, \emph{Second domain variation for problems
  with {R}obin boundary conditions}, Journal of Optimization Theory and
  Applications \textbf{167} (2015), no.~2, 430--463.

\bibitem{bobkov2020qualitative}
Vladimir Bobkov and Sergey Kolonitskii, \emph{On qualitative properties of
  solutions for elliptic problems with the $p$-{L}aplacian through domain
  perturbations}, Communications in Partial Differential Equations \textbf{45}
  (2020), no.~3, 230--252.

\bibitem{del2009some}
LEANDRO~M Del~Pezzo and J~Fern{\'a}ndez~Bonder, \emph{Some optimization
  problems for $p$-{L}aplacian type equations}, Applied Mathematics and
  Optimization \textbf{59} (2009), no.~3, 365--381.

\bibitem{Filinovskiy2014}
Alexey Filinovskiy, \emph{On the eigenvalues of a {R}obin problem with a large
  parameter}, Math. Bohem. \textbf{139} (2014), no.~2, 341--352. \MR{3238844}

\bibitem{Filinovskiy2017}
\bysame, \emph{On the asymptotic behavior of the first eigenvalue of {R}obin
  problem with large parameter}, J. Elliptic Parabol. Equ. \textbf{1} (2015),
  123--135. \MR{3403415}

\bibitem{garabedian1953convexity}
Paul~R Garabedian and Menahem Schiffer, \emph{Convexity of domain functionals},
  Tech. report, Stanford Univ CA Applied Mathematics and Statistics Labs, 1953.

\bibitem{gavitone2018first}
Nunzia Gavitone and Leonardo Trani, \emph{On the first {R}obin eigenvalue of a
  class of anisotropic operators}, Milan Journal of Mathematics \textbf{86}
  (2018), no.~2, 201--223.

\bibitem{giorgi2005monotonicity}
Tiziana Giorgi and Robert~G Smits, \emph{Monotonicity results for the principal
  eigenvalue of the generalized {R}obin problem}, Illinois Journal of
  Mathematics \textbf{49} (2005), no.~4, 1133--1143.

\bibitem{hadamard1908memoire}
Jacques Hadamard, \emph{M{\'e}moire sur le probl{\`e}me d'analyse relatif {\`a}
  l'{\'e}quilibre des plaques {\'e}lastiques encastr{\'e}es}, vol.~33,
  Imprimerie nationale, 1908.

\bibitem{ii2001placement}
E.~M. Harrell~II, Pawel Kr{\"o}ger, and Kazuhiro Kurata, \emph{On the placement
  of an obstacle or a well so as to optimize the fundamental eigenvalue}, SIAM
  Journal on Mathematical Analysis \textbf{33} (2001), no.~1, 240--259.

\bibitem{henrot2017shape}
Antoine Henrot, \emph{Shape optimization and spectral theory}, De Gruyter,
  2017.

\bibitem{kesavan2003two}
Srinivasan Kesavan, \emph{On two functionals connected to the {L}aplacian in a
  class of doubly connected domains}, Proceedings of the Royal Society of
  Edinburgh Section A: Mathematics \textbf{133} (2003), no.~3, 617--624.

\bibitem{MR0241822}
O.~A. Lady\v{z}enskaja, V.~A. Solonnikov, and N.~N. Ural'ceva, \emph{Linear and
  quasilinear equations of parabolic type}, Translations of Mathematical
  Monographs, Vol. 23, American Mathematical Society, Providence, R.I., 1968,
  Translated from the Russian by S. Smith. \MR{0241822}

\bibitem{lamberti2003differentiability}
Pier~Domenico Lamberti, \emph{A differentiability result for the first
  eigenvalue of the $p$-{L}aplacian upon domain perturbation}, Nonlinear
  analysis and applications: to V. Lakshmikantham on his 80th birthday
  \textbf{2} (2003), 741--754.

\bibitem{le2006eigenvalue}
An~L{\^e}, \emph{Eigenvalue problems for the $p$-{L}aplacian}, Nonlinear
  Analysis: Theory, Methods \& Applications \textbf{64} (2006), no.~5,
  1057--1099.

\bibitem{lieberman1988boundary}
Gary~M Lieberman, \emph{Boundary regularity for solutions of degenerate
  elliptic equations}, Nonlinear Analysis: Theory, Methods \& Applications
  \textbf{12} (1988), no.~11, 1203--1219.

\bibitem{melian2001perturbation}
Jorge~Garc{\'\i}a Meli{\'a}n and Jos{\'e}~Sabina de~Lis, \emph{On the
  perturbation of eigenvalues for the $p$-{L}aplacian}, Comptes Rendus de
  l'Acad{\'e}mie des Sciences-Series I-Mathematics \textbf{332} (2001), no.~10,
  893--898.

\bibitem{osawa1992hadamard}
Tatsuzo Osawa, \emph{The hadamard variational formula for the ground state
  value of {$ -\Delta u= \lambda| u|^{p- 1}u$}}, Kodai Mathematical Journal
  \textbf{15} (1992), no.~2, 258--278.

\bibitem{payne1957lower}
LE~Payne and HF~Weinberger, \emph{Lower bounds for vibration frequencies of
  elastically supported membranes and plates}, Journal of the Society for
  Industrial and Applied Mathematics \textbf{5} (1957), no.~4, 171--182.

\bibitem{roppongi1994hadamard}
Susumu Roppongi, \emph{The hadamard variation of the ground state value of some
  quasi-linear elliptic equations}, Kodai Mathematical Journal \textbf{17}
  (1994), no.~2, 214--227.

\end{thebibliography}
	\medskip

\medskip
\noindent
\textbf{Ardra A}:\quad
Department of Mathematics, Indian Institute of Technology Palakkad, Kerala-678557, India, \\  
	\textit{Email}: \texttt{ardra.math@gmail.com, 211814001@smail.iitpkd.ac.in}
	
\medskip
\noindent
\textbf{Mohan Mallick}:\quad
Department of Mathematics, Visvesvaraya National Institute of Technology Nagpur, Maharashtra, 440010, India, \\  
	\textit{Email}: \texttt{mohan.math09@gmail.com, mohanmallick@mth.vnit.ac.in}

\medskip
\noindent
\textbf{Sarath Sasi}:\quad
Department of Mathematics, Indian Institute of Technology Palakkad, Kerala-678557, India, \\  
	\textit{Email}: \texttt{sarath@iitpkd.ac.in}

\section*{Acknowledgment}
The authors would like to thank Dr. Vladimir Bobkov for the convergence result given in Proposition \ref{Robin limit goes to dirichlet}.

\section*{Conflict of interest}

The authors declare that they have no conflict of interest.
\end{document}